\documentclass{article}

\usepackage[english]{babel}

\usepackage[letterpaper,top=2cm,bottom=2cm,left=3cm,right=3cm,marginparwidth=1.75cm]{geometry}

\usepackage{amsmath,amssymb,amsthm}
\usepackage{graphicx}
\usepackage{tikz}
\usepackage{circuitikz}
\usetikzlibrary{arrows,positioning}
\usepackage{caption}
\usepackage{subcaption}
\usepackage{xcolor}
\usepackage[colorlinks=true, allcolors=blue]{hyperref}

\DeclareMathOperator{\ima}{im}

\newcommand{\dC}{\mathbb{C}}

\newcommand{\K}{\mathbb{K}}
\newcommand{\Hc}{\mathcal{H}}

\newcommand{\Ac}{\mathcal{A}}
\newcommand{\Bc}{\mathcal{B}}
\newcommand{\Cc}{\mathcal{C}}
\newcommand{\Dc}{\mathcal{D}}
\newcommand{\Lc}{\mathcal{L}}
\newcommand{\Mc}{\mathcal{M}}

\usepackage{calc}

\makeatletter
\newlength\@SizeOfCirc%
\newcommand{\CricArrowRight}[1]{%
    \setlength{\@SizeOfCirc}{\maxof{\widthof{#1}}{\heightof{#1}}}%
    \tikz [x=1.0ex,y=1.0ex,line width=.15ex, draw=black]%
        \draw [->,anchor=center]%
            node (0,0) {#1}%
            (0,1.2\@SizeOfCirc) arc (85:-240:1.2\@SizeOfCirc);%
}%
{\left(\begin{smallmatrix}}
{\end{smallmatrix}\right)}

\newenvironment{smallbmatrix}
{\left[\begin{smallmatrix}}
{\end{smallmatrix}\right]}

\usepackage{tikz-cd}
\usetikzlibrary{positioning,arrows,decorations.markings}
\tikzset{
block/.style={
  draw, 
  rectangle, 
  minimum height=1.5cm, 
  minimum width=2.5cm, align=center,
  fill=blue!20
  }, 
line/.style={->,>=latex'}
}
\tikzset{negated/.style={
      decoration={markings,
           mark= at position 0.5 with {
               \node[transform shape] (tempnode) {$\backslash\!\!\backslash$};
            }
       },
       postaction={decorate}
    }
}

\newtheorem{definition}{Definition}[section]

\newtheorem{proposition}[definition]{Proposition}

\theoremstyle{definition}

\title{A~geometric framework for discrete time port-Hamiltonian systems}

\author{Karim Cherifi\thanks{Institut f\"{u}r Mathematik, Technische Universität Berlin, Stra\ss e des 17.\ Juni 136, 10623 Berlin, Germany  (\texttt{cherifi@math.tu-berlin.de}).} \and Hannes Gernandt\thanks{Fraunhofer IEG, Fraunhofer Research Institution for Energy Infrastructures and Geothermal Systems IEG, Cottbus, Gulbener Straße 23, 03046 Cottbus, Germany  (\texttt{hannes.gernandt@ieg.fraunhofer.de}).}
\and Dorothea Hinsen\thanks{ Institut f\"{u}r Mathematik, Technische Universität Berlin, Stra\ss e des 17.\ Juni 136, 10623 Berlin, Germany   (\texttt{hinsen@math.tu-berlin.de}).} \and Volker Mehrmann \thanks{Institut f\"{u}r Mathematik, Technische Universität Berlin, Stra\ss e des 17.\ Juni 136, 10623 Berlin, Germany   (\texttt{hinsen@math.tu-berlin.de}).} }

\begin{document}
\maketitle

\begin{abstract}
  Port-Hamiltonian systems  provide an energy-based formulation with a model class that is closed under structure preserving interconnection. 
  For continuous-time systems these interconnections are constructed by geometric objects  called Dirac structures. In this paper, we derive this geometric formulation and the interconnection properties for scattering passive discrete-time port-Hamiltonian systems. 
\end{abstract}

\section{Introduction}
For some major classes of continuous-time (linear) control systems, in particular port-Hamiltonion (pH) systems, it is well-established that staying close to the underlying physics requires a general geometric framework. This has lead to the definition of continuous-time port-Hamiltonian systems via Dirac, Lagrange and monotone structures, see e.g. \cite{GerHR21,MehS23,SchM23,SchM18,SchM20}. In this paper this framework will be extended  to discrete-time descriptor systems \cite{BanV19,Bru09,Dai89,Meh91} of the form
\begin{align}
\label{discr_DAE}
Ex_{k+1}&=Ax_k+Bu_k,\quad y_k=Cx_k+Du_k,\quad k\geq 0,\quad x_0=x^0,
\end{align}
with sequences of vectors $x_k\in\K^n$, $u_k,y_k\in\K^m$, $k=0,1,2,\ldots$, $E,A\in\K^{n\times n}$, $B\in\K^{n\times m}$, $C\in\K^{m\times n}$. Here $\K\in\{\mathbb R,\mathbb C\}$, and we allow a possibly singular $E$.
Discrete-time descriptor systems of the form \eqref{discr_DAE}  are called \emph{scattering passive} if there exists a storage function $V:\K^n\rightarrow[0,\infty)$ satisfying $V(0)=0$ and the \emph{dissipation inequality}
\begin{equation}
 V(Ex_{k+1})-V(Ex_k)\leq \|u_k\|^2-\|y_k\|^2\quad \text{for all $k\geq 0$}   \label{eq:dsPA}
\end{equation}
for all combinations of $u_k\in\K^m$ and $x_0\in\K^n$ for which a solution of \eqref{discr_DAE} exists.

In \cite{CheGHM23} a definition of discrete-time scattering pH descriptor systems was presented for the case that the system is \emph{causal}, i.e. the solution $x_k$ at index $k$ does not depend on future inputs. (This has been characterized in \cite{CheGHM23} by the Kronecker index of the pair $(E,A)$ being at most one). In this case the system can be transformed to a reduced standard discrete-time system 
\begin{align}
\label{discr_ODE}
\tilde x_{k+1}&=\Ac \tilde x_k+\Bc u_k,\quad \tilde y_k=\Cc \tilde x_k+\Dc u_k,\quad k\geq 0,\quad \tilde x_0=\tilde x^0,
\end{align}
with sequences of vectors $\tilde x_k\in\K^r$, $u_k,\tilde y_k\in\K^m$, $k=0,1,2,\ldots$, $\Ac\in\K^{r\times r}$, $\Bc\in\K^{r\times m}$, $\Cc\in\K^{m\times r}$,
together with an algebraic equation that uniquely specifies the remaining $n-r$ components of $x_k$. A causal system of the form~\eqref{discr_DAE} is called \emph{discrete-time scattering pH system} if for the reduced standard system \eqref{discr_ODE} (denoted by $(I_r,\Ac,\Bc,\Cc,\Dc)$), there exists a Hermitian matrix $\mathcal X=\mathcal X^H>0$ such that for the norm weighted with $\textrm{diag}(\mathcal X,I_m)$ the inequality 
\begin{align}
\label{def:scat_pH}
\Big\| \begin{bmatrix}
 \Ac & \Bc\\ \Cc & \Dc
 \end{bmatrix}\Big\|_{\mathcal X} \leq 1
\end{align}
holds.
In \eqref{def:scat_pH} the weight $\mathcal X\in\K^{n\times n}$ defines the Hamiltonian $\Hc$ of the reduced system according to $\Hc(\tilde x):=\tfrac12\tilde x^H\mathcal X \tilde x$, where $\tilde x^H$ denotes the complex conjugate transpose of $\tilde x$.

In this paper, we extend and define the geometric system formulation to scattering passive discrete-time pH systems of the form \eqref{discr_ODE}. Once this is defined, we discuss the interconnection of multiple discrete-time pH systems in analogy to the continuous-time case, where it is known that the interconnection of multiple pH systems results again in a pH system. 
As special case we also discuss norm preserving interconnections.

It is shown in \cite{MehM19,MehU23} that interconnections using Dirac structures not only preserve the passivity of the interconnected system but also the pH system structure. With these results in mind, it is interesting to analyze whether this property holds for the  
scattering pH formulation of discrete-time systems using contractive interconnections. 
Here the difficulty is that we have considered only scattering pH formulations of causal descriptor systems. However, even the interconnection of two standard state-space discrete-time systems may lead to a non-causal discrete-time descriptor system. We therefore restrict ourselves to the case that in the interconnection causality is preserved.

\section{Preliminaries} 
The geometric formulation of discrete-time pH systems will be based on using subspaces $\mathcal{M}$ of $\K^n\times\K^n$ with the following additional structural properties. 
\begin{definition}
A subspace $\mathcal{M}\subseteq\K^n\times \K^n$ is called \emph{contractive} if $\|w\|\leq\|v\|$ holds for all $(v,w)\in\mathcal{M}$.
Furthermore, $\mathcal{M}$ is called \emph{maximally contractive} if $\mathcal{M}$ is not a proper subspace of a contractive subspace of $\K^n\times\K^n$. A subspace $\mathcal{M}\subseteq\K^n\times \K^n$ is called \emph{norm preserving} if $\|v\|=\|w\|$ holds for all $(v,w)\in\mathcal{M}$. 
Furthermore, $\mathcal{M}$ is called \emph{maximally norm preserving} if $\mathcal{M}$ is not a proper subspace of a norm preserving subspace of $\K^n\times\K^n$. A subspace $\mathcal{M}=\ima\begin{smallbmatrix} M_1^H\\ M_2^H
\end{smallbmatrix}$, $M_1,M_2\in\K^{n\times n}$ is called \emph{monotone} if 
\begin{align}
\label{def_mono}
M_2M_1^H+M_1M_2^H\geq 0,
\end{align}
and \emph{maximally monotone} if $\mathcal{M}$ is not a proper subspace of a monotone subspace of $\K^n\times\K^n$.
\end{definition} 

Recently, monotone subspaces were used in the continuous-time pH context \cite{GerHR21,MehS23},   as an extension of results given in \cite{MasS18}, where subspaces $\Mc$ which satisfy \eqref{def_mono} with equality were considered. In this special case, $\mathcal{M}$ is called a \emph{Dirac structure}. Closely related  to this are 
 \emph{Lagrangian subspaces} $\Lc =\ima\begin{smallbmatrix} L_1\\ L_2
\end{smallbmatrix}$, for some  $L_1,L_2\in\K^{n\times n}$ which satisfy $L_2^HL_1=L_1^HL_2$.

In the following, we will focus on contractive subspaces to describe discrete-time pH systems which can be seen as a counterpart of monotone structures for the continuous-time case. More precisely, the relation between monotone and contractive subspaces is established by a subspace variant of the well-known Cayley transformation, see also \cite{KurS07}. 

For a subspace $\mathcal{M}$ of $\K^n\times\K^n$ and for some $\alpha,\beta\in\dC\setminus\{0\}$, the \emph{Cayley transform} is defined by 
\[
\mathcal {C}_{\alpha,\beta}(\mathcal{M}) :=\{(v, w) ~\vert~ (\alpha(v+w),- \beta(v-w))\in\mathcal{M}\}.
\] 
With this, we can give the following characterization of contractive and monotone subspaces.
\begin{proposition}
\label{prop:contractive}
Let $\Mc=\ima\begin{bmatrix}M_1^H\\ M_2^H\end{bmatrix}$ for some  $M_1,M_2\in\K^{n\times r}$ be  contractive. Then the following assertions hold.
\begin{itemize}
\setlength{\itemindent}{2em}
\item[\rm (a)] $\Mc$ is contractive if and only if $M_2M_2^H-M_1M_1^H\leq 0$.
\item[\rm (b)]  For all $\alpha,\beta\in\K\setminus\{0\}$ one has $\mathcal {C}_{\alpha,\beta}(\mathcal{M})=\ima\begin{bmatrix}
\beta M_1^H+\alpha M_2^H\\ \beta M_1^H-\alpha M_2^H
\end{bmatrix}$.
\item[\rm (c)] Let $\alpha\overline{\beta}>0$. If $\Mc$ is monotone then $\mathcal{C}_{\alpha,\beta}(\mathcal{M})$ is contractive. Conversely, if $\Mc$ is contractive and if $\tfrac{\vert\alpha\vert}{\vert\beta\vert}\leq 1$,  then $\mathcal{C}_{\alpha,\beta}(\mathcal{M})$ is monotone.
\item[\rm (d)] $\Mc$ is maximally contractive if and only if $\dim \Mc=n$.
\end{itemize}
\end{proposition}
\begin{proof}
\begin{itemize}
\setlength{\itemindent}{2em}
    \item[(a)]  For all $(v,w)\in\Mc$, there exists some $z\in\K^r$ such that $(v,w)=(M_1^Hz,M_2^Hz)$ holds. Hence, $\Mc$ is contractive if and only if for all $z\in\K^r$ the inequality
\[
z^HM_2M_2^Hz=\|M_2^Hz\|^2=\|w\|^2\leq\|v\|^2=\|M_1^Hz\|^2=z^HM_1M_1^Hz
\]
holds.
\item[(b)] By definition, we have $(v,w)\in \mathcal {C}_{\alpha,\beta}(\Mc)$ if and only if  
\[
\alpha(v+w)=M_1^Hz,\quad -\beta(v-w)=M_2^Hz
\]
holds for some $z\in\K^r$. Hence
\[
v=\frac12(\alpha^{-1}M_1^H+\beta^{-1}M_2^H)z,\quad w=\frac12(\alpha^{-1}M_1^H-\beta^{-1}M_2^H)z
\]
and after multiplication with $2\alpha\beta$ we obtain the formula in (b).
\item[(c)] 
    If $\Mc$ is monotone, then $\alpha\overline{\beta}>0$ yields 
\begin{align*}
&~~~~(\beta M_1^H-\alpha M_2^H)^H(\beta M_1^H-\alpha M_2^H)-(\beta M_1^H+\alpha M_2^H)^H(\beta M_1^H+\alpha M_2^H)\\ &=-\overline{\beta}\alpha M_1M_2^H-\overline{\alpha}\beta M_2M_1^H-\overline{\beta}\alpha M_1M_2^H-\overline{\alpha}\beta M_2M_1^H\\ &=-2\alpha\overline{\beta}(M_1M_2^H+M_2M_1^H)\\ &\leq0.
\end{align*}
Hence, (a) and (b) imply that  $\mathcal{C}_{\alpha,\beta}(\Mc)$ is contractive. Conversely, if $\Mc$ is contractive, then $\alpha\overline{\beta}>0$ and $\tfrac{\vert\alpha\vert}{\vert\beta\vert}\leq 1$ imply 
    \begin{align*}
    &~~~~(\beta M_1^H+\alpha M_2^H)^H(\beta M_1^H-\alpha M_2^H)+(\beta M_1^H-\alpha M_2^H)^H(\beta M_1^H+\alpha M_2^H)\\ &=2\vert\beta\vert^2(M_1M_1^H-\tfrac{\vert\alpha\vert^2}{\vert\beta\vert^2}M_2M_2^H)\\ &\geq 2\vert\beta\vert^2(1-\tfrac{\vert\alpha\vert^2}{\vert\beta\vert^2})M_2M_2^H \\ &\geq 0. 
    \end{align*} 
\item[(d)] Observe that the image representation of the Cayley transformed relation is given by
\begin{align} \beta
\label{eq:cayley_matrix}
\begin{bmatrix}
\tfrac{\beta}{\alpha}M_1^H+M_2^H\\ \tfrac{\beta}{\alpha}M_1^H-M_2^H
\end{bmatrix}=\begin{bmatrix}
\tfrac{\beta}{\alpha}I_n&I_n\\\tfrac{\beta}{\alpha}I_n&-I_n
\end{bmatrix}\begin{bmatrix}
M_1^H\\ M_2^H
\end{bmatrix}.
\end{align}
Since the block matrix on the right-hand side of \eqref{eq:cayley_matrix} is invertible, the dimension of the subspace is preserved under the Cayley transformation. Furthermore, we know from (c) that the Cayley transform relates monotone and contractive subspaces. It is also known that a monotone subspace $\Mc\subseteq\K^n\times\K^n$ is maximal, i.e.\ it is not contained in a monotone subspace of larger dimension, if and only if $\dim\Mc=n$ holds, see \cite[Lemma 3.6]{GerHR21}. Hence, maximally contractive subspaces are mapped to maximally monotone subspaces via the Cayley transformation, where the latter subspaces are $n$-dimensional. Hence assertion (d) follows.
\end{itemize}
\end{proof}

In this section we have introduced the concept of contractive and monotone subspaces and their coordinate representations. In the next section we introduce the geometric formulation of discrete-time port-Hamiltonian systems.

\section{Geometric formulation of discrete-time port-Hamiltonian systems}\label{sec:geometric}

In this section, we introduce a geometric coordinate-free formulation of discrete-time port-Hamiltonian systems. As a motivation, we consider first the special case of pH systems without input and output variables in continuous time. The case of continuous-time \emph{dissipative Hamiltonian (dH) descriptor systems} 
\begin{align}
\label{eq:cnt_dHDAE}
\tfrac{d}{dt} Ex(t)=(J-R)Qx(t), \quad J=-J^H,\quad R=R^H\geq 0,\quad Q^HE=E^HQ
\end{align}
was investigated in \cite{MehMW18}. Based on \cite{MasS18}, it was shown in \cite{GerHR21} that a geometric formulation of \eqref{eq:cnt_dHDAE} is given by 
\begin{align}
\label{cnttime_geo_ph}
(z(t),-\dot z(t))\in\mathcal{M}^{-1}\mathcal{L}=\{(z,w) ~\vert~(w,v)\in \mathcal{M},~ (z,v)\in\mathcal{L}\},
\end{align}
where $z:[0,\infty)\rightarrow \K^n$ is a continuously differentiable function,  $\mathcal{M}$ is a monotone subspace and $\mathcal{L}$ is a Lagrange subspace.

An immediate approach to obtain a discrete-time formulation is to integrate in \eqref{cnttime_geo_ph} using the trapezoidal rule.  
 This leads to the following \emph{geometric formulation of a discrete-time scattering dH system}
\begin{align}
\label{tustin_relation}
\left(\frac{h}{2}(z_k+z_{k+1}),-(z_{k+1}-z_k)\right)\in\mathcal{M}^{-1}\mathcal{L}\quad \Longleftrightarrow \quad (z_k,z_{k+1})\in {\mathcal C}_{\frac{h}{2},1}(\mathcal{M}^{-1}\mathcal{L}).
\end{align}

If we consider for simplicity $\mathcal{L}=\{(x,x) ~|~ x\in\K^n\}$, then  $\mathcal{M}^{-1}\mathcal{L}$ is a monotone subspace, and, invoking Proposition~\ref{prop:contractive}, we find that the subspace ${\mathcal C}_{\frac{h}{2},1}(\mathcal{M}^{-1}\mathcal{L})$ in \eqref{tustin_relation} is contractive.

After this motivation, we give a geometric definition based on the analogy to continuous-time pH descriptor systems in \cite{SchM23}. 
\begin{definition}
\label{def:geom_pH}
A geometric representation of a pH system with state space $\K^n$ and external dimension $m$ is given by a triple $(\mathcal{M}, \mathcal{R})$ consisting of
\begin{itemize}
    \item a maximal norm-preserving subspace $\mathcal{N}\subseteq \K^{n+r+m}\times \K^{n+r+m}$,
    \item a maximal contractive subspace $\mathcal{C}\subseteq \K^r\times\K^r$.
\end{itemize}
By a solution of the discrete-time pH system $(\mathcal{N}, \mathcal{C})$  we understand an input-state-output trajectory $((u_k)_{k\geq 0},(x_k)_{k\geq 0}, (y_k)_{k\geq 0})$ for which there exist sequences $(f_{k,R})_{k\geq 0}$ and $(e_{k,R})_{k\geq 0}$, and such that for all $k \geq  0$ we have
\begin{align}
\label{eq:geom_incl}
(x_{k+1}, f_{k,R}, y_k, x_k, e_{k,R}, u_k)\in\mathcal{N},\quad 
(f_{k,R}, e_{k,R}) \in\mathcal{C}.
\end{align}
\end{definition}

The solutions of a geometric discrete-time pH system satisfy the following power balance equation
\[
\|x_{k+1}\|^2+\|f_{k,R}\|^2+\|y_k\|^2=\|x_k\|^2+\|e_{k,R}\|^2 +\|u_k\|^2,\quad k\geq 0.
\]  
Rearranging terms and employing the contractivity implies that
\begin{align}
\label{eq:almost_passive}
\|x_{k+1}\|^2-\|x_k\|^2\leq \|u_k\|^2-\|y_k\|^2,\quad k\geq 0,
\end{align}
which shows that the geometric definition of discrete-time scattering pH systems leads to scattering passive systems.

The geometric definition introduced in \cite{SchM23} also includes an additional Lagrange structure to generalize the concept of the Hamiltonian. An extension of Definition~\ref{def:geom_pH} in this direction is rather straight forward, but left for future work. Here one introduces Lagrangian effort variables $e_{L,k}$ which must fulfill $(x_{k+1},e_{k,L})\in \mathcal{C}_{\frac{h}{2},1}(\mathcal{L})$ for some $h>0$ and one has to replace $x_k$ in the left hand side expression in \eqref{eq:geom_incl} by $e_{k,L}$.

\section{Contractive interconnection of scattering passive systems}
It is well known that the loss-less interconnection of continuous-time pH systems using Dirac subspaces preserves the pH system structure \cite{MehM19,MehU23}. In this section, we present an analogous interconnection result for discrete-time pH systems. A related approach describing the interconnection of Dirac subspaces of scattering passive systems was given in \cite{CerSB07} and it is based on the use of the \emph{Redheffer star product} and wave variable representations of effort and flow variables.

To describe the system interconnection, we restrict ourselves to the interconnection of two scattering passive systems  given by 
\begin{align}
E_1x_{1,k+1}&=A_1x_{1,k}+B_1 u_{1,k}, &y_{1,k}&=C_1x_{1,k}+D_1u_{1,k}, \label{eq:scat_sys_1} \\ 
E_2x_{2,k+1}&=A_2x_{2,k}+B_2 u_{2,k} & y_{2,k}&=C_2x_{2,k}+D_2u_{2,k} \label{eq:scat_sys_2}
\end{align}
for coefficient matrices of appropriate sizes and where the inputs and outputs of the systems are partitioned as
\[
u_{i,k}=\begin{bmatrix}
    u_{i,k}^1\\ u_{i,k}^2
\end{bmatrix},\quad  y_{i,k}=\begin{bmatrix}
    y_{i,k}^1\\ y_{i,k}^2
\end{bmatrix},\quad B_i=[B_i^1,B_i^2],\quad C_i=\begin{bmatrix}
    C_i^1\\ C_i^2
\end{bmatrix},\quad D_i=\begin{bmatrix}
    D_i^{11}& D_i^{12}\\ D_i^{21}& D_i^{22}
\end{bmatrix}, \quad i=1,2,
\] 
where the components $u_{i,k}^1$ and $y_{i,k}^1$ are available for coupling.
In the following we consider a contractive interconnection $\Mc$ of the scattering passive systems \eqref{eq:scat_sys_1} and \eqref{eq:scat_sys_2}  given by 
\begin{align}
    \label{eq:inter_M}
(u_{1,k}^1,u_{2,k}^1,y_{1,k}^1,y_{2,k}^1)\in\Mc,\quad \left\|\begin{bmatrix}
    y_{1,k}^1\\ y_{2,k}^1
\end{bmatrix}\right\|\leq \left\|\begin{bmatrix}
    u_{1,k}^1\\ u_{2,k}^1
\end{bmatrix}\right\|,\quad k\geq 0.
\end{align}
A special case of such an interconnection \eqref{eq:inter_M} is given by 
\begin{align}
\label{redhfr_connect}
y_{1,k}^1=u_{2,k}^1,\quad u_{1,k}^1=y_{2,k}^1
\end{align}
and was used in \cite{CerSB07} for the interconnection of Dirac structures. It is closely related to the Redheffer star product which was studied in \cite{MisW86,Red62}. 

Then,we have the following main result on structure preserving interconnection. To show that the scattering pH structure is preserved, we restrict to the case $E_1=E_2=I_n$ because of spatial limitations. However, the result can be generalized for descriptor systems having index at most one.
\begin{proposition}
Consider scattering passive systems \eqref{eq:scat_sys_1} and \eqref{eq:scat_sys_2}, i.e.\ \eqref{eq:dsPA} holds, together with an interconnection via \eqref{eq:inter_M}. Then the following holds:
\begin{itemize}
    \item[\rm (a)] The interconnected system given by \eqref{eq:scat_sys_1}, \eqref{eq:scat_sys_2} and \eqref{eq:inter_M}  is scattering passive.
    \item[\rm (b)] If the systems \eqref{eq:scat_sys_1} and \eqref{eq:scat_sys_2} are discrete time scattering pH, i.e.\ \eqref{def:scat_pH} holds, and let $E_1=E_2=Id$ and $I-D_{1}^{11}D_2^{11}$ be invertible, then the interconnected system~\eqref{redhfr_connect}  is equivalent to a discrete-time pH system.
    \item[\rm (c)] Let \eqref{eq:scat_sys_1} and \eqref{eq:scat_sys_2} be scattering pH, i.e.\ \eqref{def:scat_pH} holds for some positive definite $\mathcal{X}_1,\mathcal{X}_2>0$, and fulfill $E_1=E_2=Id$. If $u_{i,k}=u_{i,k}^1$ and $y_{i,k}=y_{i,k}^1$ holds for $i=1,2$ and all $k\geq 0$, then the interconnection \eqref{redhfr_connect} leads to a contractive closed loop system of the form 
    \[
 \begin{bmatrix}
x_{1, k+1}\\x_{2,k+1}
\end{bmatrix}=\hat A\begin{bmatrix}
x_{1,k}\\x_{2,k}
\end{bmatrix},\quad \|\hat A z\|_{\hat{\mathcal{X}}}\leq \|z\|_{\hat{\mathcal{X}}},\quad\quad  \text{where}\quad \|z\|_{\hat{\mathcal{X}}}:=\left\| \begin{bmatrix}\mathcal{X}_1&0\\0& \mathcal{X}_2\end{bmatrix}z\right\|.
    \]
\end{itemize}
\end{proposition}
\begin{proof}
 If we consider a kernel representation $\Mc=\ker[M_1^1,M_1^2,M_2^1,M_2^2]$ for some matrices of appropriate sizes, then the interconnected system can be written as \hspace{-3cm}
\begin{align} 
 \mathbf{A}&=\begin{bmatrix}
 A_1&0&B_1^1&0&0&0\\0&A_2&0&B_2^1&0&0\\ C_1^1&0&D_1^{11}&0&-I&0\\0&C_2^1&0&D_2^{11}&0&-I\\0&0&M_1^1&M_1^2&M_2^1&M_2^2   
\end{bmatrix},& \mathbf{B}&=\begin{bmatrix}
    B_1^2&0\\0&B_2^2\\D_1^{12}&0\\0& D_2^{12}\\0&0
\end{bmatrix},&
\mathbf{x}_k&=\begin{bmatrix}
    x_{1,k}\\x_{2,k}\\u_{1,k}^1\\ u_{2,k}^1\\ y_{1,k}^1\\ y_{2,k}^1
\end{bmatrix}, &\mathbf{E}&=\begin{bmatrix}
    E_1&0&0\\0&E_2&0\\0&0&0
\end{bmatrix}& &\nonumber \\  \mathbf{C}&=\begin{bmatrix}
    C_1^2&0&D_1^{21}&0&0&0 \\0&C_2^2&0&D_2^{21}&0&0
\end{bmatrix}, &\mathbf{D}&=\begin{bmatrix}
    D_{1}^{22}&0\\0&D_2^{22}
\end{bmatrix}, &\mathbf{u}_k&=\begin{bmatrix}
    u_{1,k}^2\\u_{2,k}^2
\end{bmatrix} , &\mathbf{y}_k&=\begin{bmatrix}
    y_{1,k}^2\\y_{2,k}^2
\end{bmatrix}.
\end{align}
Since the systems \eqref{eq:scat_sys_1} and \eqref{eq:scat_sys_2} are assumed to be scattering passive, there exists storage functions $V_1$ and $V_2$ such that the following holds
\begin{align}
    \label{eq:two_diss_ineq}
V_i(E_ix_{i,k+1})-V_i(E_ix_{i,k})\leq \|u_{i,k}\|^2-\|y_{i,k}\|^2,\quad i=1,2.
\end{align}
Furthermore, by the contractive interconnection \eqref{eq:inter_M} we have
\[
\|y_{1,k}^1\|^2+\|y_{2,k}^1\|^2=\left\|\begin{bmatrix}
    y_{1,k}^1\\ y_{2,k}^1
\end{bmatrix}\right\|^2\leq \left\| \begin{bmatrix}
    u_{1,k}^1\\ u_{2,k}^1
\end{bmatrix}\right\|^2=\|u_{1,k}^1\|^2+\|u_{2,k}^1\|^2.
\]
Adding the inequalities \eqref{eq:two_diss_ineq} and using $\mathbf{V}(\mathbf{E}\mathbf{x}_k):=V_1(E_1x_{1,k})+V_2(E_2x_{2,k})$ yields
\begin{align*}
    \mathbf{V}(\mathbf{E}\mathbf{x}_{k+1})-\mathbf{V}(\mathbf{E}\mathbf{x}_{k})\leq 
 \sum_{i,j=1}^2\|u_{i,k}^j\|^2-\|y_{i,k}^j\|^2 \leq \|u_{1,k}^2\|^2+\|u_{2,k}^2\|^2-\|y_{1,k}^2\|^2-\|y_{2,k}^2\|^2
 =\|\mathbf{u}_k\|^2-\|\mathbf{y}_k\|^2,
\end{align*}
which is the dissipation inequality for the interconnected system with respect to the scattering supply rate and the quadratic storage function $\mathbf{V}$. Hence, the interconnected system is scattering passive, which proves (a).

We continue with the proof of (b) and consider the interconnection  \eqref{redhfr_connect} of the systems \eqref{eq:scat_sys_1} and \eqref{eq:scat_sys_2} which is given by
\begin{align}
\mathbf{A}&=\begin{bmatrix}
 A_1&0&B_1^1&0\\0&A_2&0&B_2^1\\ C_1^1&0&D_1^{11}&-I\\0&C_2^1&-I&D_2^{11}
\end{bmatrix},&  \mathbf{B}&=\begin{bmatrix}
    B_1^2&0\\0&B_2^2\\D_1^{12}&0\\0& D_2^{12}
\end{bmatrix},&
\mathbf{x}_k&=\begin{bmatrix}
    x_{1,k}\\x_{2,k}\\u_{1,k}^1\\ u_{2,k}^1
\end{bmatrix},& \mathbf{y}_k=\begin{bmatrix}
    y_{1,k}^2\\y_{2,k}^2
\end{bmatrix}, \nonumber\\ \mathbf{C}&=\begin{bmatrix}
    C_1^2&0&D_1^{21}&0 \\0&C_2^2&0&D_2^{21}
\end{bmatrix}, &\mathbf{D}&=\begin{bmatrix}
    D_{1}^{22}&0\\0&D_2^{22}
\end{bmatrix}, &  \mathbf{E}&=\begin{bmatrix}
    E_1&0&0\\0&E_2&0\\0&0&0
\end{bmatrix}, &\mathbf{u}_k=\begin{bmatrix}
    u_{1,k}^2\\u_{2,k}^2
\end{bmatrix}.
\end{align}

It was shown in \cite{MisW86} that the invertibility of $I-D_1^{11}D_2^{11}$ is equivalent to that of $I-D_2^{11}D_1^{11}$ and also to the following condition
\[
\ker\begin{bmatrix}
    D_1^{11}& -I \\-I & D_2^{11}
\end{bmatrix}=\{0\}.
\]
Therefore, we can rewrite
\[
\begin{bmatrix} u_{1,k}^1\\ u_{2,k}^1\end{bmatrix}=-\begin{bmatrix}
     D_1^{11}&-I\\  -I&D_2^{11}
\end{bmatrix}^{-1}\begin{bmatrix}
     C_1^1&0\\  0&C_2^1
\end{bmatrix}\begin{bmatrix}x_{1,k}\\ x_{2,k}\end{bmatrix}.
\]

Hence, the interconnected system is equivalent to 
\begin{align*}
\hat{\mathbf{A}}&:=\begin{bmatrix}
     A_1&0\\  0&A_2
\end{bmatrix}-\begin{bmatrix}
     B_1^1&0\\  0&B_2^1
\end{bmatrix} \begin{bmatrix}
     D_1^{11}&-I\\  -I&D_2^{11}
\end{bmatrix}^{-1}\begin{bmatrix}
     C_1^1&0\\  0&C_2^1
\end{bmatrix}
,&\hat{\mathbf{E}}:&=\begin{bmatrix}
    E_1&0\\0&E_2
\end{bmatrix}, \\
\hat{\mathbf{B}}&:=\begin{bmatrix}
    B_1^2&0\\0& B_2^2
\end{bmatrix}-\begin{bmatrix}
     D_1^{11}&-I\\  -I&D_2^{11}
\end{bmatrix}^{-1}\begin{bmatrix}
    D_1^{12}&0\\0& D_2^{12}
\end{bmatrix}, &\hat{\mathbf{x}}_k&=\begin{bmatrix}
    x_{1,k}\\x_{2,k}
\end{bmatrix},
\\ \hat{\mathbf{C}}&:=\begin{bmatrix}
    C_1^2&0 \\0&C_2^2
\end{bmatrix}-\begin{bmatrix}
    D_1^{21}&0\\ 0&D_2^{21}
\end{bmatrix}\begin{bmatrix}
     D_1^{11}&-I\\  -I&D_2^{11}
\end{bmatrix}^{-1}\begin{bmatrix}
     C_1^1&0\\  0&C_2^1
\end{bmatrix},&\hat{\mathbf{y}}_k&=\begin{bmatrix}
    y_{1,k}^2\\y_{2,k}^2
\end{bmatrix},\\ \hat{\mathbf{D}}&:=\begin{bmatrix}
    D_{1}^{22}&0\\0&D_2^{22}
\end{bmatrix}-\begin{bmatrix}
    D_1^{21}&0\\0& D_2^{21}
\end{bmatrix}\begin{bmatrix}
     D_1^{11}&-I\\  -I&D_2^{11}
\end{bmatrix}^{-1}\begin{bmatrix}
    D_1^{12}&0\\0&D_2^{12}
\end{bmatrix},&\hat{\mathbf{u}}_k&=\begin{bmatrix}
    u_{1,k}^2\\u_{2,k}^2
\end{bmatrix}.
\end{align*}
Assertion (a) yields that the interconnected system is scattering passive and since \eqref{def:scat_pH} is valid for each of the systems, the following holds
\[
\|\hat{\mathbf{A}}\hat{\mathbf{x}}_k+\hat{\mathbf{B}}\hat{\mathbf{u}}_k\|_{\hat{\mathcal{X}}}^2-\|\hat{\mathbf{x}}_k\|_{\hat{\mathcal{X}}}^2=\|\hat{\mathbf{x}}_{k+1}\|_{\hat{\mathcal{X}}}^2-\|\hat{\mathbf{x}}_k\|_{\hat{\mathcal{X}}}^2\leq\|\hat{\mathbf{u}}_k\|^2-\|\hat{\mathbf{y}}_k\|^2=\|\hat{\mathbf{u}}_k\|^2-\|\hat{\mathbf{C}}\hat{\mathbf{x}}_k+\hat{\mathbf{D}}\hat{\mathbf{u}}_k\|^2,
\]
and therefore
\[
\left\|\begin{bmatrix}
    \hat{\mathbf{A}}& \hat{\mathbf{B}}\\ \hat{\mathbf{C}}& \hat{\mathbf{D}}
\end{bmatrix} \begin{bmatrix} \hat{\mathbf{x}}_k\\ \hat{\mathbf{u}}_k\end{bmatrix}\right\|_{\hat{\mathcal{X}}}^2\leq \left\|\begin{bmatrix} \hat{\mathbf{x}}_k\\ \hat{\mathbf{u}}_k\end{bmatrix}\right\|_{\hat{\mathcal{X}}}^2 \quad \text{or, equivalently,} \quad 
\left\|\begin{bmatrix}
    \hat{\mathbf{A}}& \hat{\mathbf{B}}\\ \hat{\mathbf{C}}& \hat{\mathbf{D}}
\end{bmatrix}\right\|_{\hat{\mathcal{X}}}\leq 1,
\]
which means that the interconnected system $(\hat{\mathbf{A}},\hat{\mathbf{B}},\hat{\mathbf{C}},\hat{\mathbf{D}})$ is scattering pH according to  \eqref{def:scat_pH}. 

We proceed with the proof of (c). In this case the scattering pH systems are given by 
\[
\begin{bmatrix}x_{1,k+1}\\y_{1,k}^1\end{bmatrix}=\begin{bmatrix}
   A_1 & B_1\\ C_1 & D_1 
\end{bmatrix}\begin{bmatrix}
    x_{1,k}\\ u_{1,k}^1
\end{bmatrix}, \quad \begin{bmatrix}x_{2,k+1}\\y_{2,k}^1\end{bmatrix}=\begin{bmatrix}
  A_2 & B_2\\ C_2 & D_2 
\end{bmatrix}\begin{bmatrix}
    x_{2,k}\\ u_{2,k}^1 \end{bmatrix},
\]
and since $\left\|\begin{smallbmatrix}A_i & B_1\\ C_i & D_i\end{smallbmatrix}\right\|_{\mathcal{X}_i}\leq 1$ holds for $i=1,2$, we obtain the closed loop system 
\begin{align}
    \label{eq:closedloop}
\begin{bmatrix}
x_{1, k+1}\\x_{2,k+1}
\end{bmatrix}=\begin{bmatrix}
    A_1+B_1D_2(I-D_1D_2)^{-1}C_1 & B_1C_2+B_1D_2(I-D_1D_2)^{-1}D_1C_2 \\ B_2(I-D_1D_2)^{-1}C_1&A_2+B_2(I-D_1D_2)^{-1}D_1C_2
\end{bmatrix}\begin{bmatrix}
x_{1,k}\\x_{2,k}
\end{bmatrix}.
\end{align}
Adding the dissipation inequalities of the systems implies 
\begin{align*}
\|x_{1,k+1}\|_{\mathcal{X}_1}^2+\|x_{2,k+1}\|_{\mathcal{X}_2}^2&\leq \|x_{1,k}\|_{\mathcal{X}_1}^2+\|u_{1,k}\|^2-\|y_{1,k}\|^2+\|x_{2,k}\|_{\mathcal{X}_1}^2+\|u_{2,k}\|^2-\|y_{2,k}\|^2\\ &\leq \|x_{1,k}\|_{\mathcal{X}_1}^2+\|x_{2,k}\|_{\mathcal{X}_2}^2,
\end{align*}
which implies that the coefficient matrix of the closed loop system \eqref{eq:closedloop} is a contraction.

\end{proof}

\section{Conclusion}
In this paper, we presented a geometric approach to scattering passive discrete time pH systems which is based on contractive subspaces. It also discussed the interconnection of discrete time pH systems and we showed that the interconnection of two pH systems results in another pH system. We also had a closer look at the Redheffer interconnection and showed that the interconnection in this case also results in a pH system. 

\section*{Acknowledgement}
The work of K.~Cherifi has been supported by ProFIT (co-financed by the Europäischen Fonds für regionale Entwicklung (EFRE)) within the WvSC project: EA 2.0 - Elektrische Antriebstechnik (project No. 10167552). 
The work of H.~Gernandt has been supported by the Deutsche Forschungsgemeinschaft (DFG, German Research Foundation) within the Priority Programme 1984 ``Hybrid and multimodal energy systems'' (Project No.~ 361092219) and the Wenner-Gren Foundation.  The work of D.~Hinsen and V.~Mehrmann has been supported by the Deutsche Forschungsgemeinschaft (DFG, German Research Foundation) CRC 910 \emph{Control of self-organizing nonlinear systems: Theoretical methods and concepts of application}: Project No.~163436311 and by Bundesministerium für Bildung und Forschung (BMBF)  EKSSE: Energieeffiziente Koordination und Steuerung des Schienenverkehrs in Echtzeit (grant no. 05M22KTB).

\bibliographystyle{plain}
\bibliography{pamm-tpl}

\end{document}